\documentclass[a4paper,12pt]{amsproc}

\usepackage{xargs}
\usepackage{mathtools, todonotes}
\usepackage{amsmath,amssymb,amsfonts,amsthm}

\usepackage[colorlinks]{hyperref}
\usepackage{amsrefs}
\usepackage{enumerate}

\newcommandx{\set}[2][2=\empty]{\{#1\ifx#2\empty\else\,|\,#2\fi\}}% set builder notation
\newcommandx{\gensubgrp}[2][2=\empty]{\langle#1\ifx#2\empty\else\,|\,#2\fi\rangle}% generated subgroup
\newcommandx{\gennorsubgrp}[2][2=\empty]{\langle\!\langle#1\ifx#2\empty\else\,|\,#2\fi\rangle\!\rangle}% generated normal subgroup
\newcommandx{\gensubsp}[2][2=\empty]{\langle#1\ifx#2\empty\else\,|\,#2\fi\rangle}% generated subspace
\newcommand{\norm}[1]{\lvert#1\rvert}% norm of an element
\newcommand{\complex}{\mathbb C}% the complex numbers
\newcommand{\freegrp}{\mathbf F}% free group
\newcommand{\ints}{\mathbb{Z}}% integers
\DeclareMathOperator{\codim}{codim}% codimension
\DeclareMathOperator{\rk}{rk}% the rank of a matrix
\DeclareMathOperator{\GL}{GL}% general linear group
\DeclareMathOperator{\SL}{SL}% special linear group
\DeclareMathOperator{\PGL}{PGL}% general linear group
\DeclareMathOperator{\PSL}{PSL}% special linear group
\DeclareMathOperator{\Alt}{Alt}% alternating group
\DeclareMathOperator{\Sym}{Sym}% alternating group
\newcommand{\finfield}{\mathbb{F}}% finite field
\DeclareMathOperator{\diam}{diam}% diameter
\DeclareMathOperator{\id}{id}% the identity
\newcommand{\floor}[1]{\left\lfloor#1\right\rfloor}% floor operator
\newcommand{\ceil}[1]{\left\lceil#1\right\rceil}% ceiling operator
\newcommand{\abs}[1]{\lvert#1\rvert}% absolute value
\newcommand{\nats}{\mathbb{N}}% natural numbers
\newcommand{\zerovecsp}{\mathbf{0}}
\DeclareMathOperator{\dist}{dist}% distance of points
\newcommand{\C}{\mathbf{C}}% centralizer
\newcommand{\trivgrp}{\mathbf{1}}% trivial group
\newcommand{\reals}{\mathbb{R}}% real numbers
\title[Word maps for linear groups]{Non-singular word maps for linear groups}
\author{Henry Bradford}
\address{H.~Bradford, Univ.\ of Cambridge, Cambridge
CB3 0WB, Unites Kingdom}
\email{hb470@cam.ac.uk}
\author{Jakob Schneider}
\address{J.~Schneider, TU Dresden, 01062 Dresden, Germany}
\email{jakob.schneider@tu-dresden.de}
\author{Andreas Thom}
\address{A.~Thom, TU Dresden, 01062 Dresden, Germany}
\email{andreas.thom@tu-dresden.de}

\theoremstyle{plain}
\newtheorem{definition}{Definition}
\newtheorem{theorem}{Theorem}
\newtheorem{lemma}{Lemma}
\newtheorem{corollary}{Corollary}
\newtheorem{conjecture}{Conjecture}
\newtheorem{question}{Question}

\theoremstyle{definition}
\newtheorem{remark}{Remark}

\begin{document}
\begin{abstract}
We study the word image of words with constants in $\GL(V)$ and show that it is large provided the word satisfies some natural conditions on its length and its critical constants.
    
There are various consequences: We prove that for every $l \geq 1$, there are only finitely many pairs $(n,q)$ such that the length of the shortest non-singular mixed identity $\PSL_n(q)$ is bounded by $l$. We generalize the Hull--Osin dichotomy for highly transitive permutation groups to linear groups over finite fields. Finally, we show that the rank limit of $\GL_n(q)$ for $q$ fixed and $n \to \infty$ is mixed identity free.
\end{abstract}

\maketitle

\tableofcontents

\section*{Introduction} \label{intro}

Denote by $\freegrp_r=\gensubgrp{x_1,\ldots,x_r}$ the free group on $r$ generators. Let $G$ be a group and let $w\in G\ast\freegrp_r$ be a word with $r$ variables $x_1,\dots,x_r$ and constants in $G$. We say that $w$ is a \emph{mixed identity} for $G$ if $w$ is non-trivial and $w(h_1,\dots,h_r)=1_G$ for all $h_1,\dots,h_r \in G.$
Here we write $w(h_1,\ldots,h_r)$ for the image $\varphi(w)$ of $w$ under the unique homomorphism $\varphi\colon G\ast\freegrp_r\to G$ which maps $x_i\mapsto h_i$ for all $i\in\set{1,\ldots,r}$ and fixes $G$ elementwise. On $G \ast \freegrp_r$ we consider the word length, which assigns length one to the generators $x_1^{\pm1},\dots,x_r^{\pm1}$ and length zero to elements of $G$. In this setting, the \emph{augmentation map} $\varepsilon\colon G\ast\freegrp_r\to\freegrp_r$ denotes the unique homomorphism which maps all $G\ni g\mapsto 1_{\freegrp_r}$ to the identity and fixes each element of $\freegrp_r$. The elements in the kernel of $\varepsilon$ are called \emph{singular} and $\varepsilon(w)$ is called the \emph{content} of $w$.

In this article we continue the study of non-singular word maps on finite groups, which we started in \cite{schneiderthom2022word}, where the second and third named authors treated the case of symmetric and alternating groups.
One of the main results in \cite{schneiderthom2022word} was that any non-singular mixed identity of $S_n$ is of length at least $\Omega(\log(n)/\log(\log(n))).$ This is in contrast to the length of singular mixed identities, which can be of bounded length. See \cite{bradfordschneiderthom2023length} for more information on mixed identities in this context. Non-singular word maps for compact Lie groups where studied in \cite{klyachkothom2017new} and for linear algebraic groups in \cite{gordeevkunyavskiiplotkin2018wordmaps}.

A \emph{law} or an \emph{identity} of a group $G$ is a non-trivial word $w \in \freegrp_r$ such that $w(h_1,\dots,h_r)=1_G$ for all $h_1,\dots,h_r \in G$. We say that $w \in \freegrp_r$ is a \emph{coset identity} if there exists a subgroup of finite index $H$ and cosets $g_1H,\dots,g_r H$, such that $w(h_1,\dots,h_r)=1_G$ for all $h_i \in g_i H, i \in \{1,\dots,r\}$. It is a major open problem to decide if every finitely generated pro-finite group that admits a coset identity also admits an identity, see for example \cite{mazurovkhukhro2014kourovka}*{Problem 12.95} and the work of Larsen--Shalev \cite{larsenshalev2016probabilistic}*{Problem 3.1}. Note that every coset identity yields a non-singular mixed identity
$$w'(x_1,\dots,x_r)\coloneqq w(g_1 x_1^{d},\dots,g_rx_r^d)$$
for $d=[G:H]!$.
Hence, a potential strategy to answer Larsen and Shalev's question is to study the length of the shortest non-singular mixed identity of a finite group $G$ and compare it to the length of the shortest identity. It can be shown that the answer to Larsen and Shalev's question is positive if these quantities can be bounded in terms of each other.

\begin{question} Does the existence of a non-singular mixed identity $w \in G \ast \freegrp_r$ for a finite group $G$ imply the existence of a law for $G$ of length bounded only in terms of the length of $w$?
\end{question}

We are able to prove that the length of the shortest non-singular mixed identity for $\PSL_n(q)$ tends to infinity as $n$ or $q$ increase. Note, that the minimal length of laws for groups $\PSL_n(q)$ has been studied in \cite{bradfordthom2018laws}. Hence, this contributes another class of cases where this question has a positive answer, see Corollary \ref{cor:bound}. 

In view of the corresponding result of Jones \cite{jones1974varieties} on laws for non-abelian finite simple groups, we put forward the following conjecture:

\begin{conjecture}
Let $l \in \nats$. The number of isomorphism classes of non-abelian finite simple groups satisfying a non-singular mixed identity of length $\leq l$ is finite.
\end{conjecture}

Similarly to our results for symmetric groups, our analysis of word maps allows also for detailed information on the image of word maps, singular or non-singular, as long as the lengths of the critical constants (see Section~\ref{sec:basics} for a definition) are under certain control. Our main result is stated as Theorem~\ref{thm:vecsp}. We have various applications of this result in analogy to results on permutation actions. In particular, we have estimates on the diameter and dimension of the image. 

It was shown by Gordeev--Kunyavskii--Plotkin in \cite{gordeevkunyavskiiplotkin2018wordmaps} that non-singular word maps on linear algebraic groups with generic constants are dominant. As a consequence of our results, we can prove a non-trivial lower bound on the dimension of the image in case of ${\rm SL}_n$ and non-singular word maps with arbitrary constants, see Theorem~\ref{thm:alggrp_lowerbound}.

In Section~\ref{sec:hullosin}, we can prove an analogue of the Hull--Osin dichotomy \cite{hullosin2016transitivity}*{Theorem~1.6} for highly linearly transitive groups, which in its original form says that a highly transitive permutation group either contains a locally finite, highly transitive, normal subgroup or it is mixed identity free. We prove a similar result for highly linearly transitive groups acting on a countably infinite dimensional vector space over a finite field, see Theorem~\ref{thm:hull_osin_lingrp}. There is also a quantitative version of this, generalizing a result of Le Boudec--Matte Bon \cite{leboudecmattebon2022triple}*{Proposition A.1}, see Theorem~\ref{thm:mattebon_lingrp}. 

It is a consequence of the original Hull--Osin dichotomy that full groups of ergodic, probability measure preserving, measurable equivalence relations on a non-atomic probability space are mixed identity free, see the discussion in Section~\ref{sec:aq}. Moreover, it follows from Popa's work \cite{popa1995free} that the unitary group of a II$_1$-factor does not satisfy a mixed identity. No such approach can be taken to understand the mod-$p$ analogue of the unitary group of the hyperfinite II$_1$-factor as introduced by Carderi and the third author in \cite{carderithom2018exotic}. However, we give a direct proof that this group is also mixed identity free, see Theorem~\ref{thm:no_ex_mxd_id}.

Finally, this paper might be the starting point of a more systematic study of mixed identities for oligomophic permutation groups such as ${\rm Sym}(\nats)$ of $\GL(\finfield_q^{\oplus \omega})$. Note that ${\rm Aut}(\mathbb Q,<)$ satisfies a mixed identity as consequence of work of Zarzycki \cite{zarzycki2010limits}, while the automorphism group of the Rado graph and the universal Urysohn metric space does not, see Etedadialiabadi--Gao--Le Ma\^{i}tre-Melleray \cite{etedadialiabadigaolemaitremelleray2021dense}. See also Remark \ref{rem:olig} for further discussion. 

\section{Basic notions}
\label{sec:basics}

Let's recall some basics where we follow the notation introduced in \cite{schneiderthom2022word}. Let $G$ be a group. Throughout the entire article, let 
$$
w=c_0x_{\iota(1)}^{\varepsilon(1)}c_1\cdots c_{l-1}x_{\iota(l)}^{\varepsilon(l)}c_l\in G\ast\freegrp_r=G\ast\gensubgrp{x_1,\ldots,x_r}
$$ 
be a fixed word with constants $c_j\in G$ (for $j\in\set{0,\ldots,l}$) with $\varepsilon(j)\in\set{\pm1}$ and $\iota(j)\in\set{1,\ldots,r}$ (for $j\in\set{1,\ldots,l}$). Subsequently, $H\leq G$ will be the group whose elements we plug into $w$. Set $J_0(w)\coloneqq\set{j\in\set{1,\ldots,l-1}}[\iota(j)\neq \iota(j+1)]$, $J_+(w)\coloneqq\set{j\in\set{1,\ldots,l-1}}[\iota(j)=\iota(j+1)\text{ and }\varepsilon(j)=\varepsilon(j+1)]$, and $J_-(w)\coloneqq\set{j\in\set{1,\ldots,l-1}}[\iota(j)=\iota(j+1)\text{ and }\varepsilon(j)=-\varepsilon(j+1)]$. We call the elements of $J_-(w)$ the \emph{critical indices} and the group elements $c_j$ for $j\in J_-(w)$ the \emph{critical constants} of $w$. Furthermore, we assume that the above expression for $w$ is \emph{reduced}, i.e.\ $c_j\notin\C_G(H)\setminus\trivgrp$ for $j\in J_0(w)\cup J_+(w)$ and $c_j\notin\C_G(H)$ for all $j\in J_-(w)$. Here $\C_G(H)$ denotes the \emph{centralizer} of $H$ in $G$. 
%The advantage of this definition is, when $H=G$ (so $\C_G(H)=\Z(G)$) is a \emph{quasisimple} group, then a non-trivial reduced mixed identity for $H$ descends to a non-trivial reduced mixed identity for the simple group $\overline{H}=H/\Z(H)$. 
We call the word $w$ \emph{strong} if it has no critical constants, i.e.\ $J_-(w)=\emptyset$. And we call it \emph{cyclically strong} if $w$ is strong and, in addition, it does not hold that $\iota(1)=\iota(l)$ and $\varepsilon(1)=-\varepsilon(l)$. Finally, we write $\norm{w}\coloneqq l$ for the \emph{length} of the word $w$.

Throughout, we will need some measure of distance on our groups. That is why we define the notion of a \emph{seminorm} on a group. For a group $G$, this is a mapping $\norm{\bullet}\colon G\to I$, where $I$ is either the unit interval $[0,1]\subseteq\reals$ or the set $\nats\cup\set{\infty}$, such that $\norm{g}\geq0$ (\emph{non-negativity}), $\norm{gh}\leq\norm{g}+\norm{h}$ (\emph{subadditivity}), and $\norm{g}=\norm{g^{-1}}$ (\emph{homogeneity}) for all $g,h\in G$. Each seminorm induces a left invariant \emph{pseudometric} $\dist\colon G\times G\to I$ by $\dist(g,h)\coloneqq\norm{g^{-1}h}.$ 
%(since $\dist(fg,fh)=\norm{(fg)^{-1}fh}=\norm{g^{-1}h}=\dist(g,h)$ for all $f,g,h\in G$). 
A seminorm becomes a norm if $\norm{g}=0$ implies $g=1_G$ for all $g\in G$ (\emph{positivity}). In this case, it induces a metric. For each seminorm $\norm{\bullet}$, the set $G_0\coloneqq\norm{\bullet}^{-1}(0)\leq G$ is a subgroup. The seminorm is called \emph{invariant} if it is invariant under conjugation, i.e.\ $\norm{g^h}=\norm{h^{-1}gh}=\norm{g}$ for all $g,h\in G$. In this case, $G_0\trianglelefteq G$ is a normal subgroup; then $\norm{\bullet}$ descends to a norm on $G/G_0$. Then, the corresponding pseudometric $\dist\colon G\times G\to I$ will be \emph{bi-invariant}, i.e.\ $\dist(fg,fh)=\dist(g,h)=\dist(gf,hf)$ for all $f,g,h\in G$. In this article, all seminorms will be invariant, whence all pseudometrics will be bi-invariant. For every seminorm $\norm{\bullet}\colon G\to I$ and a normal subgroup $N\trianglelefteq G$, we define the \emph{quotient seminorm} $\norm{\bullet}\colon G/N\to I$ by $\norm{\overline{g}}\coloneqq\dist(g,N)=\inf\set{\dist(g,n)}[n\in N]$ for all $g\in G$.

For $w\in G\ast\freegrp_r$ and a fixed seminorm on $G$, we define the \emph{critical length}
$$ \norm{w}_{\rm crit}\coloneqq\min\set{\diam(G),\norm{c_j}}[j\in J_-(w)],
$$ 
where the diameter is measured with respect to the fixed seminorm. Besides the word length of a word, the critical length turns out to be the most important quantity controlling lower bounds on the size of the word image.

\section{Multiply linearly transitive groups}\label{sec:lin_grps}

Let $V$ be a vector space over the field $K$ and of finite or countably infinite dimension. We only consider the cases when $K=\finfield_q$ is a finite field, or $K=\complex$ are the complex numbers. Write $\GL(V)$ for the general linear group of $V$. 
Denote by $\norm{\bullet} \colon\GL(V)\to\nats\cup\set{\infty}$  the \emph{projective rank seminorm}, which is defined by 
$$
\norm{g} \coloneqq\min_{\lambda\in K^\times}\rk(g-\lambda 1_V)\in\nats\cup\set{\infty}
$$
for $g\in\GL(V)$.
It descends to a \emph{norm}, which we also denote by $$\norm{\bullet} \colon\PGL(V)\to\nats\cup\set{\infty}.$$ This induces the \emph{projective rank pseudometric} on $\GL(V)$ and a \emph{projective rank metric} on $\PGL(V)$, both denoted by $\dist $.

\begin{definition} Let $V$ be a vector space of finite or countably infinite dimension and let $k \in \mathbb N$.
A linear group $H\leq\GL(V)$ is called \emph{linearly $k$-transitive} if for all $k$-element sets of linearly independent vectors $\set{u_1,\ldots,u_k}$ and $\set{w_1,\ldots,w_k}$ there is a map $h\in H$ such that $u_i.h=w_i$ for all $i\in\set{1,\ldots,k}$.
\end{definition}

Here is our main technical result:

\begin{theorem}\label{thm:vecsp}
Let $V$ be some vector space of finite or countably infinite dimension over the field $K$. If $K$ is finite, we allow $\dim(V)$ to be infinite, whereas for $K=\complex$, we require it to be finite. Let $D,d,l,r \in \mathbb N$ and let the word 
$$
w=x_{\iota(1)}^{\varepsilon(1)}c_1\cdots c_{l-1}x_{\iota(l)}^{\varepsilon(l)}\in\GL(V)\ast\freegrp_r
$$ 
be reduced of length $l \geq 2$. 
Let $H \leq \GL(V)$ be a linearly $D$-transitive group. Let $\set{u_1,\ldots,u_d}$ and $\set{w_1,\ldots,w_d}$ be $d$-element sets of linearly independent vectors. If $\iota(1)=\iota(l)$ and $\varepsilon(1) = -\varepsilon(l)$, then assume in addition that the linear spans of $\set{u_1,\ldots,u_d}$ and $\set{w_1,\ldots,w_d}$ intersect trivially. Then, if $d$ satisfies
$$
d\leq \frac1l \cdot \min\{\norm{w}_{\rm crit}-1, D\} ,
$$
there exist $h_1,\ldots,h_r\in H$ such that $u_i.w(h_1,\ldots,h_r)=w_i$ for $i\in\set{1,\ldots,d}$.
\end{theorem}

\begin{proof}
We set $U= \langle u_1,\ldots,u_d \rangle$ and $W\coloneqq \langle w_1,\ldots,w_d\rangle.$
In order to prove the theorem, we will define vectors $v_{i,j}^{-\varepsilon(j)}\in V$
for $(i,j) \in \{1,\dots,d\} \times \{1,\dots,l\}$ and elements $h_1,\ldots,h_r\in H$ such that

\begin{gather*}
u_1=v_{1,1}^{\varepsilon(1)}\stackrel{h^{\varepsilon(1)}_{\iota(1)}}{\mapsto}
v_{1,1}^{-\varepsilon(1)}\stackrel{c_1}{\mapsto} v_{1,2}^{\varepsilon(2)}\stackrel{h^{\varepsilon(2)}_{\iota(2)}}{\mapsto}\cdots
\stackrel{h^{\varepsilon(l-1)}_{\iota(l-1)}}{\mapsto}
v_{1,l-1}^{-\varepsilon(l-1)}\stackrel{c_{l-1}}{\mapsto}
v_{1,l}^{\varepsilon(l)}\stackrel{h^{\varepsilon(l)}_{\iota(l)}}{\mapsto}v_{1,l}^{-\varepsilon(l)}=w_1\\
\vdots\\
u_d=v_{d,1}^{\varepsilon(1)}\stackrel{h^{\varepsilon(1)}_{\iota(1)}}{\mapsto}
v_{d,1}^{-\varepsilon(1)}\stackrel{c_1}{\mapsto} v_{d,2}^{\varepsilon(2)}\stackrel{h^{\varepsilon(2)}_{\iota(2)}}{\mapsto}\cdots
\stackrel{h^{\varepsilon(l-1)}_{\iota(l-1)}}{\mapsto}
v_{d,l-1}^{-\varepsilon(l-1)}\stackrel{c_{l-1}}{\mapsto}
v_{d,l}^{\varepsilon(l)}\stackrel{h^{\varepsilon(l)}_{\iota(l)}}{\mapsto}v_{d,l}^{-\varepsilon(l)}=w_d.
\end{gather*}

This requirement will determine vectors $v^{\varepsilon(j)}_{i,j} \in V$ by the equation $v_{i,j}^{\varepsilon(j)}\coloneqq v_{i,j-1}^{-\varepsilon(j-1)}.c_{j-1}$ for $2 \leq j \leq l$ and $v^{\varepsilon(1)}_{i,1} \coloneqq u_i.$
The choice of the vectors $v^{-\varepsilon(j)}_{i,j} \in V$ proceeds by induction on the lexicographic order $\leq$ on the set $\set{1,\ldots,d} \times \set{1,\ldots,l}$. During the procedure, we want to ensure that the sets
$\set{v_{i,j}^\varepsilon}[i\in\set{1,\ldots,d},j\in\set{1,\ldots,l},\iota(j)=k]$ are linearly independent for all $\varepsilon \in \{\pm1\}$ and $k \in \{1,\dots,r\}$ and we can put $v_{i,l}^{-\varepsilon(l)}\coloneqq w_i$ at the end of the procedure.

Once this is done, we can choose $h_k \in H$ such that 
$v^{+}_{i,j}.h_k = v^{-}_{i,j}$ whenever $\iota(j)=k$ and the picture is complete. For this, it is enough that $H$ is linearly $dl$-transitive which we assumed to be the case. Thus, in order to finish the proof, it only remains to carry out the inductive process and keep track of the linear independence.

Before we start with the induction, let's first describe some auxiliary objects which are needed in the inductive procedure. Suppose that all $v^{\varepsilon}_{i',j'}$ with $(i',j')\leq (i,j)$ are already defined. We define subspaces for $i\in\set{1,\ldots,d}$, $j\in\set{1,\ldots,l}$, $k\in\set{1,\ldots,r}$, and $\varepsilon \in \{\pm\}$.
\begin{align*}
V_{i,j,k}^\varepsilon &\coloneqq\gensubsp{v_{i',j'}^\varepsilon}[(i',j')\leq(i,j), \iota(j')=k]\\
&\leq\sum_{k=1}^r{V_{i,j,k}^{\varepsilon}}=\gensubsp{v_{1,1}^\varepsilon,\ldots,v_{1,l}^\varepsilon,\ldots,v_{i,1}^\varepsilon,\ldots,v_{i,j}^\varepsilon}\eqqcolon V_{i,j}^\varepsilon \leq V,
\end{align*}
so that $\dim(V_{i,j}^\varepsilon)\leq(i-1)l+j$. We also set $V_{i,0,k}^\varepsilon\coloneqq V_{i-1,l,k}^\varepsilon$ for $i\in\set{2,\ldots,d}$, $k\in\set{1,\ldots,r}$, $\varepsilon \in \{\pm\}$.
For $\varepsilon\in\set{\pm}$ and $\iota \in\set{1,\ldots,r}$, we need to introduce the subspaces:
\begin{align*}
V_\iota^\varepsilon\coloneqq\begin{cases}
U & \text{if }(\iota,\varepsilon)=(\iota(1),\varepsilon(1))\text{ and } (\iota,\varepsilon)\neq(\iota(l),-\varepsilon(l))\\
W & \text{if }(\iota,\varepsilon)\neq(\iota(1),\varepsilon(1))\text{ and } (\iota,\varepsilon)=(\iota(l),-\varepsilon(l))\\
U+W & \text{if }(\iota,\varepsilon)=(\iota(1),\varepsilon(1))\text{ and } (\iota,\varepsilon)=(\iota(l),-\varepsilon(l))\\
\zerovecsp &\text{otherwise.}
\end{cases}
\end{align*}
which we will need later in the proof.

We will now describe how to choose the vectors $v_{i,j}^{-\varepsilon(j)}$.
In the $(i,j)$th step of the induction with $j<l$, all the vectors $v_{i',j'}^\varepsilon$ for $(i',j')<(i,j)$ and $\varepsilon\in\set{\pm}$, and the vector $v_{i,j}^{\varepsilon(j)}$, are already defined. Hence the subspace $V_{i',j',k}^\varepsilon$ for $(i',j')<(i,j)$ for $k\in\set{1,\ldots,r}$ and $\varepsilon\in\set{\pm}$ and the subspace $V_{i,j,\iota(j)}^{\varepsilon(j)}=V_{i,j-1,\iota(j)}^{\varepsilon(j)}+\gensubsp{v_{i,j}^{\varepsilon(j)}}$ are already defined.
%For $k\in\set{1,\ldots,r}\setminus\set{\iota(j)}$, we have by the above definition that $V_{i,j,k}^\varepsilon=V_{i,j-1,k}^\varepsilon$ for $\varepsilon\in\set{\pm}$, so these objects are also already fixed. 
The vector $v_{i,j}^{-\varepsilon(j)}$ is not yet defined and our aim is to construct it now.

In order to retain the linear independence condition, we need to ensure that $$v_{i,j}^{-\varepsilon(j)}\notin V_{i,j-1,\iota(j)}^{-\varepsilon(j)} + V_{\iota(j)}^{-\varepsilon(j)} \quad \mbox{and} \quad v_{i,j+1}^{\varepsilon(j+1)} = v_{i,j}^{-\varepsilon(j)}.c_j \notin V_{i,j,\iota(j+1)}^{\varepsilon(j+1)} + V_{\iota(j+1)}^{\varepsilon(j+1)}.$$

%At first, we assume that $l=2$. Then we must have $j=1$, so we obtain that 
%$$
%v_{i,1}^{-\varepsilon(1)}\notin V_{i,0,\iota(1)}^{-\varepsilon(1)}+ %V_{\iota(1)}^{-\varepsilon(1)},
%$$ but $-\varepsilon(j)=-\varepsilon(1)\neq\varepsilon(1)$ so $U_{\iota(j)}^{-\varepsilon(j)}=U_{\iota(1)}^{-\varepsilon(1)}=\zerovecsp$. The remaining subspace 
%$$
%V_{i,0,\iota(1)}^{-\varepsilon(1)}+ W_{\iota(1)}^{-\varepsilon(1)}\leq V_{d-1,l}^{-\varepsilon(1)}+ W_{\iota(1)}^{-\varepsilon(1)}=V_{d-1,l}^{-\varepsilon(1)}+\gensubsp{w_d}
%$$ 
%has dimension of at most $2(d-1)+1=dl-1$. Therefore we may assume that $l\geq3$. 
For the first constraint $v_{i,j}^{-\varepsilon(j)}\notin V_{i,j-1,\iota(j)}^{-\varepsilon(j)}+ V_{\iota(j)}^{-\varepsilon(j)}$, note that this subspace is contained in $V_{d,l-2}^{-\varepsilon(j)}+\gensubsp{w_d}$ and so its dimension is at most $dl-1$. 
The case $l=2,j=1$ needs a special argument, but note that $V_{\iota(1)}^{-\varepsilon(1)}\leq W$ by its definition and hence we obtain the same bound also in this case.
In order to study the second constraint, we will now distinghish two cases.

\emph{Case~1.} Assume $j\in J_0(w)\cup J_+(w)$. We must ensure in addition that 
$$
v_{i,j}^{-\varepsilon(j)}.c_j\notin V_{i,j,\iota(j+1)}^{\varepsilon(j+1)}+ V_{\iota(j+1)}^{\varepsilon(j+1)}.
$$
Again, this condition excludes a subspace of dimension at most $dl-1$. 
Indeed, we have
$$
V_{i,j,\iota(j+1)}^{\varepsilon(j+1)}+V_{\iota(j+1)}^{\varepsilon(j+1)}\leq \begin{cases} V_{d,l-2}^{\varepsilon(j+1)}+\gensubsp{w_d} & 1 \leq j \leq l-2 \\
V^{\varepsilon(l)}_{d,l-1} & j=l-1,\end{cases},
$$ 
since $V_{\iota(l)}^{\varepsilon(l)}\leq U$ by definition in the second case, which yields the required bound on the dimension in each case.

Dealing with both constraints, we must have
$$
v_{i,j}^{-\varepsilon(j)}\notin (V_{i,j-1,\iota(j)}^{-\varepsilon(j)}+ V_{\iota(j)}^{-\varepsilon(j)})\cup(V_{i,j,\iota(j+1)}^{\varepsilon(j+1)}+ V_{\iota(j+1)}^{\varepsilon(j+1)}).c_j^{-1}
$$
and this is a union of two subspaces of dimension at most $dl-1$. Since $dl-1<dl<\norm{w'}_{\rm crit}\leq\dim(V)$, these subspaces have codimension at least $2$ and we find an appropriate choice for $v_{i,j}^{-\varepsilon(j)}$ in this case.

\emph{Case~2.} Assume now that $j\in J_-(w)$. We have to ensure that
\begin{align*}
v_{i,j}^{-\varepsilon(j)}.c_j\notin V_{i,j,\iota(j+1)}^{\varepsilon(j+1)} + V_{\iota(j+1)}^{\varepsilon(j+1)}
&=V_{i,j-1,\iota(j)}^{\varepsilon(j+1)}+\gensubsp{v_{i,j}^{\varepsilon(j+1)}}+ V_{\iota(j+1)}^{\varepsilon(j+1)}\\
&=V_{i,j-1,\iota(j)}^{-\varepsilon(j)}+\gensubsp{v_{i,j}^{-\varepsilon(j)}}+ V_{\iota(j+1)}^{\varepsilon(j+1)}.
\end{align*}

In order to get this, it is enough to ensure that for each $\lambda \in K$
$$
v_{i,j}^{-\varepsilon(j)}.(c_j-\lambda)\notin V_{i,j-1,\iota(j)}^{-\varepsilon(j)}+V_{\iota(j+1)}^{\varepsilon(j+1)}.
$$ 
Note that this subspace is again of dimension bounded by $dl-1$, where the case $l=2$ requires an additional argument that we omit. We further distinguish three more cases:

\emph{Case~2.1:} Let's assume that $\dim(V)=n<\infty$ and let's first deal with the case $K=\finfield_q$. Let $\lambda_k\in\finfield_q^\times$ for $k\in\set{1,\ldots,m}$ be the eigenvalues of the linear map $c_j$ and $d_k$ their respective geometric multiplicity. Then for $\lambda\notin\set{\lambda_1,\dots,\lambda_m}$, $c_j-\lambda$ is regular, so the condition above excludes a union of $q-m$ subspaces of dimension at most $dl-1$. Thus, we need to exclude at most $(q-m)q^{dl-1}$ vectors.

For $\lambda=\lambda_k$ we have to exclude the inverse image of $V_{i,j-1,\iota(j)}^{-\varepsilon(j)}+V_{\iota(j+1)}^{\varepsilon(j+1)}$ under the map $c_j - \lambda.$ Thus, to deal with $\lambda_k$, we need to exclude at most $q^{d_k}q^{dl-1}$ vectors. 
Hence, when dealing with all $\lambda \in \finfield_q$, we need to avoid
at most $$
(q-m)q^{dl-1}+\sum_{k=1}^m{q^{dl-1}q^{d_k}}
$$
vectors. This covers the second constraint and we need to exclude at most $q^{dl-1}$ vectors for the first constraint. 
Hence, in order to show that there is a good choice for $v_{i,j}^{-\varepsilon(j)}$ it remains to show that
$$
(q-m+1)q^{dl-1}+\sum_{k=1}^m{q^{dl-1}q^{d_k}}<q^n-1.
$$
By assumption $n-d_k \geq |c_j| \geq |w|_{\rm crit}>dl.$
This implies that
\begin{eqnarray*}
(q-m+1)q^{dl-1}+\sum_{k=1}^m{q^{dl-1}q^{d_k}} &\leq&
(q-m+1)q^{n-2}+ m q^{n-2}< q^n-1.
\end{eqnarray*}
Hence, we have shown that there is a good choice for $v_{i,j}^{-\varepsilon(j)}$.

\emph{Case~2.2:} The proof for $K=\mathbb C$ is similar. Consider 
$$V_{\lambda} \coloneqq \left\{v \in V \mid v.(c_j-\lambda)\in V_{i,j-1,\iota(j)}^{-\varepsilon(j)}+V_{\iota(j+1)}^{\varepsilon(j+1)} \right\}$$ and note that the estimate $n-d_k \geq |c_j| \geq |w|_{\rm crit}>dl$ is still valid and provides a bound of $n-2$ on the dimension of each $V_{\lambda_k}$ for $\lambda_k$ an eigenvalue of $c_j$. In order to deal with the other choices of $\lambda$, i.e.\ $\lambda \coloneqq D = \mathbb C \setminus \{\lambda_1,\dots,\lambda_m\},$ note that $V_{\lambda}$ is a $1$-dimensional algebraic family of subspaces of codimension at least $2$. Thus, $\bigcup_{\lambda \in D} V_{\lambda}$ is of codimension at least $1$. Thus, we have to avoid finitely many algebraic subsets of lower dimension, which is of course possible.

\emph{Case~2.3:} If $K=\finfield_q$ and $\dim(V)=\infty$, there are still only finitely elements $\lambda \in \finfield_q$ to consider and the inverse image of 
$V_{i,j-1,\iota(j)}^{-\varepsilon(j)}+V_{\iota(j+1)}^{\varepsilon(j+1)}$ under the map $c_j-\lambda$ is still proper. In fact, together with the first constraint there are at most $q+1$ of these subspaces to consider and they are all of codimension at least $2$. This implies that their union has a non-trivial complement from which we can choose  $v_{i,j}^{-\varepsilon(j)}$. Thus we are done for this case as well.

Thus, this procedure can be carried out until we arrive at $j=l$ and then we put $v_{i,l}^{-\varepsilon(l)}\coloneqq w_i.$
This finishes the proof of the main technical result.
\end{proof}

%\begin{remark}
%Theorem~\ref{thm:vecsp} probably holds for any field $K$ and any infinite dimension of $V$, however one would need a better notion of a \emph{large} set in the final Step~2.2.
%\end{remark}

\begin{remark} \label{rmk:mult_trans}
For finite permutation groups $H \leq S_n$, 
it is a well-known consequence of the CFSG that the only examples (other than $A_n$ and $S_n$) of $k$-transitive groups for $k\geq 4$ are the four Matthieu groups $M_{11}, M_{12}, M_{23}$, and $M_{24}$ in degrees $n=11,12,23$ and $24$. 
There is a similar classification of 
multiply linearly transitive subgroups of 
$\GL_n(q)$, following from the work of Hering \cite{hering1985transitive}. 
Indeed the linearly $1$-transitive finite groups 
(that is, those acting transitively on the set of non-zero vectors in $V=\mathbb{F}_q ^n$) 
are already rather restricted: 
apart from a finite list of sporadic examples, 
and overgroups of $\SL_ n(q)$, 
they all preserve either a non-trivial 
symplectic form on $V$ or a proper field extension  
(meaning a structure on $V$ of an $(n/m)$-dimensional $\mathbb{F}_{q^m}$-vector space, for some proper divisor $m$ of $n$), both of which are incompatible with linear $2$-transitivity. 
The sporadic cases include one linearly $3$-transitive example (an action of $A_7$ on $\mathbb{F}_2 ^4$); none of the others are linearly $2$-transitive \cite{lee2023email}. 

%We expect but are not aware of a similar classification of linearly $k$-transitive finite groups $H \leq \GL_n(q)$ for $k$ large.
\end{remark}

The rest of the paper consists of consequences that we will explain in some detail.

\section{Diameter and dimension of the word image}

\subsection{Diameter estimates}

From Theorem~\ref{thm:vecsp} we get the following corollary on the diameter bound of the word image. 

\begin{corollary}\label{cor:diam_wrd_img_vecsp} Let $n \in \mathbb N$.
Let $V$ be an $n$-dimensional vector space over $K=\finfield_q$ or $\complex$ and $w\in\GL(V)\ast\freegrp_r$ be a reduced word of length $l \geq 2$. Then the word image $w(\SL(V)^r)$ has diameter at least
$$
\diam(w(\SL(V)^r))\geq\frac{\norm{w}_{\rm crit}}{l}-1
$$
in the projective rank seminorm.
\end{corollary}

Equivalently,
\begin{equation}\label{eq:ineq}
\norm{w}(\diam(w(\SL(V)^r))+1)\geq\norm{w}_{\rm crit}.
\end{equation}
Let's now assume that $w$ is non-singular and try to get a lower bound on the diameter of the word image only depending on the word length and not the critical length. The argument follows closely the corresponding argument for permutation groups in \cite{schneiderthom2022word}. If we remove the smallest critical constant $c_j$ in $w$ and perform the corresponding cancellation, we get a word $w'$ such that 
$$\dist(w(h_1,\ldots,h_r),w'(h_1,\ldots,h_r))\leq\dist(c_j,1)=\norm{w}_{\rm crit}$$ for all $h_1,\ldots,h_r\in H$ by bi-invariance of the rank metric. 
Hence 
$$
\abs{\diam(w(\SL(V)^r))-\diam(w'(\SL(V)^r))}\leq 2\norm{w}_{\rm crit}.
$$
so by \eqref{eq:ineq} we get
\begin{align*}
\diam(w'(\SL(V)^r))&\leq\diam(w(\SL(V)^r))+2\norm{w}_{\rm crit}\\
&\leq\diam(w(\SL(V)^r))+2\norm{w}(\diam(w(\SL(V)^r))+1).
\end{align*}
Therefore, we obtain
$$
\diam(w'(\SL(V)^r))+1\leq (1+2\norm{w})(\diam(w(\SL(V)^r))+1)
$$
Hence, if we iterate these reductions and $w=w_0,\ldots,w_m$ is a chain of words such that $w_j$ is an reduction of $w_{j-1}$ ($1\leq j\leq m$) and $w_m$ is \emph{strong}, then we get
$$
\diam(w_j(\SL(V)^r))+1\leq (1+2\norm{w_{j-1}})(\diam(w_{j-1}(\SL(V)^r))+1)
$$
for all $j=1,\ldots,m$, and by iterating we obtain
\begin{align*}
\diam(w_m(\SL(V)^r))+1 &\leq (\diam(w_0(\SL(V)^r))+1)\prod_{j=0}^{m-1}(1+2\norm{w_j})\\
&\leq(\diam(w_0(\SL(V)^r))+1)(1+2\norm{w_0})^m\\
&\leq(\diam(w(\SL(V)^r))+1)(1+2\norm{w})^{\floor{\norm{w}/2}}.
\end{align*}
This is since $\norm{w_{j-1}}\geq\norm{w_j}-2$ so that $m\leq\floor{\norm{w}/2}$. Hence by Inequality~\eqref{eq:ineq} we get
\begin{align*}
n=\norm{w_m}_{\rm crit}&\leq (\diam(w_m(\SL(V)^r))+1)\norm{w_m}\\
&\leq (\diam(w(\SL(V)^r))+1)(1+2\norm{w})^{\floor{\norm{w}/2}}\norm{w}.
\end{align*}
This gives
$$
\frac{1}{(1+2\norm{w})^{\floor{\norm{w}/2}}\norm{w}}\leq\frac{\diam(w(\SL(V)^r))+1}{n}.
$$
In particular, this implies that when $\norm{w}$ is bounded and $n\to\infty$, then also $\diam(w(\SL(V)^r))\to\infty$.

\begin{corollary} Let $q$ be a prime power.
Let $w$ be a non-singular mixed identity for ${\rm PSL}_n(q)$. Then, we have
$$|w| = \Omega\left(\frac{\log(n)}{\log\log(n)}\right).$$
\end{corollary}

\begin{corollary} \label{cor:bound}
For every $l \geq 1$, there are only finitely many pairs $(n,q)$ such that the length of the shortest non-singular mixed identity $\PSL_n(q)$ is bounded by $l$.
\end{corollary}
    
\begin{proof}
It is known from \cite{bradfordschneiderthom2023length} that any mixed identity for $\PSL_n(q)$ is of length $\Omega(q)$, where the implied constant is independent of $n$. This and the previous corollary imply the claim.
\end{proof}

\begin{remark}
By the proof of \cite{schneiderthom2022word}*{Corollary~3}, there exists a singular word $v\in\PSL_n(q) \ast\gensubgrp{x}$ of length $O(q^{3n^2})$ with $v(g)=g^{-1}$ for all $g \in \PSL_n(q)$. Clearly, $w=vx$ is a non-singular word with the property $w(g)=1$ for all $g\in \PSL_n(q)$ and content $\varepsilon(w)=x$. This shows that there is no reason to assume that non-singularity by itself or just the length of the content of a word would be enough to give a lower bound on the diameter of the word image.
\end{remark}

\subsection{Esimates of the dimension}

In this section, we follow the notation from the work of Gordeev--Kunyavskii--Plotkin \cite{gordeevkunyavskiiplotkin2018wordmaps} where we stick to the special case $G=\SL_n$. Their main result is \cite{gordeevkunyavskiiplotkin2018wordmaps}*{Theorem~1.1}, which states that for a fixed non-singular equation 
$$
w=c_0x_{\iota(1)}^{\varepsilon(1)}c_1\cdots c_{l-1}x_{\iota(l)}^{\varepsilon(l)}c_l\in G\ast\freegrp_r=G\ast\gensubgrp{x_1,\ldots,x_r}
$$
with $(c_0,\dots,c_l) \in {\rm SL}_n(\mathbb C)^{l+1}$, there is a Zariski dense set of choices for the constants, where image is full-dimensional. This is essentially a consequence of the fact that the maximal dimension is generically attained and the image is full-dimensional by an old result of Borel \cite{borel1983on} for the trivial choice of constants. However, their result says nothing about the worst case scenario.

\begin{theorem} \label{thm:alggrp_lowerbound}
Let $w=c_0x_{\iota(1)}^{\varepsilon(1)}c_1\cdots c_{l-1}x_{\iota(l)}^{\varepsilon(l)}c_l\in \GL_n(\mathbb C) \ast\freegrp_r$ be reduced of length $l \geq 2$. Then, we have
$$\dim w(\SL_n(\mathbb C)) \geq\floor{\frac{\norm{w}_{\rm crit}-1}{l}}\cdot n.$$
\end{theorem}
\begin{proof}
Note that $d=\floor{l^{-1}\norm{w}_{\rm crit}-1}$ is a valid choice for $d$ in Theorem~\ref{thm:vecsp}. Hence, by linear $d$-transitivity of the image of $w$, we get that the space of $(n \times d)$-matrices of rank $d$ is a subquotient of the image of $w$. This finishes the proof since its dimension is $dn$.
\end{proof}
\section{A Hull--Osin type result for linear groups} \label{sec:hullosin}
    
The following dichotomy theorem is due to Hull and Osin \cite{hullosin2016transitivity}*{Theorem~1.6}. Let's recall it quickly before we explain our generalization to the linear case.
    
\begin{theorem}[Hull--Osin]\label{thm:hull_osin_perm_grps}
Let $\Omega$ be a countably infinite set and let $H\leq\Sym(\Omega)$ be a highly transitive permutation group. Then either
$$\Alt(\Omega)\trianglelefteq H \quad \mbox{or} \quad
 \mbox{$H$ has no mixed identities.}$$
\end{theorem}

There is also a quantitative version of this theorem due to Le Boudec--Matte Bon \cite{leboudecmattebon2022triple}*{Proposition~A.1}:
    
\begin{theorem}[Le Boudec--Matte Bon]
Let $k\in\ints_{\geq1}$ and $H\leq\Sym(\Omega)$ be a $k$-transitive permutation group which satisfies a mixed identity of length $l$. Then either
$$\Alt(\Omega)\trianglelefteq H \quad \mbox{or} \quad
 k<l.$$
\end{theorem}    

\begin{remark}
\label{rem:olig}
Now we come back to the discussion of mixed identities for oligomorphic groups at the end of the introduction. In view of \cite{leboudecmattebon2022triple}*{Proposition A.1}, an interesting case to study is the automorphism group of the random $k$-uniform hypergraph for $k \geq 3$, which acts $(k-1)$-transitively but not $k$-transitively in its natural action. It is an open problem if a dense subgroup could have a highly transitive action and maybe this is obstructed by the existence of a mixed identity. This provides an interesting test case for \cite{hullosin2016transitivity}*{Question 6.1}.
\end{remark}
    
For the statement of the Hull--Osin-type result in the linear case, we need the following definition: For a countably infinite dimensional vector space $V$ over $\finfield_q$, we define the normal subgroup $\GL_{\rm fin}(V)\coloneqq\set{g\in\GL(V)}[\norm{g} <\infty]$ of $\GL(V)$. This group consists of all the elements $g\in\GL(V)$ such that there is a subspace $W\leq V$ of finite codimension with $w.g=\lambda w$ for all $w\in W$ and a fixed scalar $\lambda\in\finfield_q^\times$.
    
Before we state the linear version of Theorem~\ref{thm:hull_osin_perm_grps}, we need the following preparatory lemma.
    
\begin{lemma}\label{lem:hull_osin_help}
    Let $g\in\GL_{\rm fin}(V)$ act as the scalar $\lambda$ on the subspace $W\leq V$ of finite codimension $n\in\nats$ and let $U\leq V$ be a complement of $W$, i.e.\ $U\oplus W=V$. Then $U'\coloneqq U+U.g$ is a $g$-invariant subspace of dimension $\leq 2n$. There is a complement $W'\leq W$ of $U'$ in $V$.
\end{lemma}
    
\begin{proof}
    Take $u'=u+w\in U'$ arbitrarily for $u\in U$, $w\in W$. Then $w=u'-u\in U'+U=U'$, hence $u'.g=u.g+w.g=u.g+\lambda w\in U.g+U'=U'$, showing that $U'$ is $g$-invariant. Define $W'$ to be a complement to $U'\cap W$ in $W$. Then $U'\cap W'=U'\cap W\cap W'=\zerovecsp$ by definition and $U'+W'=U'+(U'\cap W)+W'=U'+W=V$, so $U'\oplus W'=V$. This completes the proof.
\end{proof}
    
Here comes the analogue of Theorem~\ref{thm:hull_osin_perm_grps} for linear groups over finite field.
    
\begin{theorem} \label{thm:hull_osin_lingrp}
Let $V$ be a countably infinite dimensional vector space over $\finfield_q$ and let $H\leq\GL(V)$ be a highly linearly transitive linear group. Then either
\begin{enumerate}[(i)]
    \item $H \cap \GL_{\rm fin}(V)$ is a highly transitive, locally finite, normal subgroup, or
    \item $H$ has no mixed identities.
\end{enumerate}
\end{theorem}

We will first prove a quantitative form, generalizing the result of Le Boudec--Matte Bon to linear groups.

\begin{theorem} \label{thm:mattebon_lingrp}
Let $V$ be a countably infinite dimensional vector space over a finite field. Let $k\in\ints_{\geq1}$ and let $H\leq\GL(V)$ be a subgroup which is linearly $k$-transitive and admits a mixed identity of length $l$. Then, at least one of the following cases hold:
\begin{enumerate}[(i)]
\item $H \cap \GL_{\rm fin}(V)$ is a locally finite, linearly $(\ceil{k/2}-1)$-transitive, normal subgroup of $H$, or
\item $k\leq 2l.$
\end{enumerate}
\end{theorem}
\begin{proof}
The inequality in Case~(ii) can be proven using Theorem~\ref{thm:vecsp} provided all non-trivial elements of $H$, and hence critical constants in a potential mixed identity $w \in H\ast\freegrp_r$, are of length at least $k/2$. Indeed, $k>2l$ implies $\norm{w}_{\rm crit} -1 \geq \ceil{k/2}-1\geq l$ and we can apply Theorem~\ref{thm:vecsp} with $d\coloneqq 1$ and $D\coloneqq k$. Hence, any such $w$ of length at most $l$ is not a mixed identity, contradicting an assumption in the statement of the theorem. We conclude that $k \leq 2l$ must hold.

Hence, in order to show Case~(i), we may suppose that there exists a critical constant $c \in H$ of non-zero length less than $k/2$. We may assume $k \geq 3.$
%Note that this already implies that $k\in\ints_{\geq3}$; but in the opposite case, (ii) is fulfilled.
Using Lemma \ref{lem:hull_osin_help}, we get for some $\lambda\in\finfield_q^{\times}$ that there exists a vector space $V_1$ of codimension $k$ with complement $V_0\cong\finfield_q^k$, such that $c$ has a matrix of the form
$$
    c= \begin{pmatrix} c' & 0 \\ 0 & \lambda 1_{V_1} \end{pmatrix},
$$
with respect to the decomposition $V=V_0 \oplus V_1$, where $c'\in\GL_k(q)$ is a non-scalar matrix. Now, since $H$ is linearly $k$-transitive, it contains elements of the form
$$g= \begin{pmatrix} g' & * \\ 0 & g'' \end{pmatrix}.$$
for arbitrary $g' \in \GL_k(q).$
Thus, since $\SL_k(q)$ is quasi-simple and $c'$ is non-scalar, the normal subgroup of $H$ generated by $c$ will contain elements of the form as above with $g'\in\SL_k(q)$ arbitrary. Together with the $k$-linear transitivity of $H$, this shows that the normal subgroup generated by $c$ is $k'$-linearly transitive for $k'\coloneqq \ceil{k/2}-1$. Indeed, any two linearly independent sets $\set{\alpha_1,\dots,\alpha_{k'}}$ and $\set{\beta_1,\dots,\beta_{k'}}$ can be moved into $V_0$ by some element $h\in H$. Now, since $2k'<k$, there exists some element  $g'\in\SL(V_0)\cong\SL_k(q)$ exchanging the two sets $\set{\alpha_1.h,\dots,\alpha_{k'}.h}$ and $\set{\beta_1.h,\dots,\beta_{k'}.h}$. Let $g\in\gennorsubgrp{c}$, which contains $g'$ in its upper left corner. Then, this shows that
$\alpha_i.hgh^{-1}=\beta_i$ for all $i\in\set{1,\dots,k'}.$ It is clear that the normal subgroup generated by $c$ is contained in $H\cap\GL_{\rm fin}(V)$ and that this group is locally finite. This finishes the proof.
\end{proof}

\begin{proof}[Proof of Theorem~\ref{thm:hull_osin_lingrp}] Since $H$ is highly linearly transitive, we can apply Theorem~\ref{thm:mattebon_lingrp} for any $k\in\ints_{\geq1}$. In particular, the existence of a mixed identity implies Case~(i) in Theorem~\ref{thm:mattebon_lingrp} for arbitrary $k$. Hence, $H \cap \GL_{\rm fin}(V)$ is a highly transitive, locally finite, normal subgroup of $H$.

However, the statement of Theorem~\ref{thm:hull_osin_lingrp} is a bit stronger since it says that Cases~(i) and (ii) are mutually exclusive. We will now show that Case~(i) arises whenever $H$ contains an element $h$ such that $\norm{h} <\infty$, i.e.\ whenever $H\cap\GL_{\rm fin}(V)$ is non-trivial. Let $W\leq V$ be a subspace of codimension $n$ such that $h$ acts as multiplication by the fixed scalar  $\lambda\in\finfield_q^\times$ on $W$. For any $g\in\GL(V)$, the element $h^g\in\GL_{\rm fin}(V)$ acts as $\lambda\id_{W.g}$ on $W.g$. Hence $[h,g]$ acts as the identity on $W_0\coloneqq W\cap W.g$. Then $\codim(W_0)\leq 2n$. By Lemma~\ref{lem:hull_osin_help}, we find $W'\leq W_0$ and $U'$ a $[h,g]$-invariant complement of $W'$, such that $\dim(U')=\codim(W')\leq 4n$. Then $[h,g]$ can be seen as an element of $\GL(U')\times\set{1_{W'}}\leq\GL_{4n}(q)\times\trivgrp$, so $[h,x]^e\in H\ast\gensubgrp{x}=H\ast\freegrp_1$ is a mixed identity for $H$, where $e$ denotes the exponent of $\GL_{4n}(q)$.
 \end{proof}
 
\section{No mixed identities for $A(q)$}
\label{sec:aq}
It was proved in \cite{bradfordschneiderthom2023nonsol}*{Proof of Theorem~1.10} that the topological full group of a faithful and minimal, continuous action of a group on a Cantor set is mixed identity free. This is essentially a direct application of the Hull--Osin theorem to the action of the full group on an infinite orbit. A similar reasoning can be applied to the study of full groups of an ergodic, probability measure preserving equivalence relation on a non-atomic standard probability space.
In this section, we want to prove a linear version of this, where a certain group $A(q)$, constructed in \cite{carderithom2018exotic}, replaces the full group. 

Let's briefly recall the construction of $A(q)$, which resembles a construction of the full group of the hyperfinite equivalence relation. Let $q$ be a prime power as usual.
For $n\in\nats$, consider $\GL_{2^n}(q)$ and the diagonal embeddings
$$
\iota_n\colon \GL_{2^n}(q)\to \GL_{2^{n+1}}(q); h\mapsto\begin{pmatrix}
h & 0\\
0 & h
\end{pmatrix}.
$$
Write $\GL_{2^\omega}(q)=\bigcup_{n\in\nats}\GL_{2^n}(q)$ for the inductive limit of the countable chain of inclusions $(\GL_{2^n}(q),\iota_n)_{n\in\nats}$. Clearly, this is a countable, locally finite group. Note that the normalized rank metric is well-defined on $\GL_{2^\omega}(q)$. We denote its Cauchy completion by $A(q)$. Note that we could have started with $\SL_{2^n}(q)$ and arrived at the same completion. Furthermore, $A(q)/Z$, where $Z$ denotes its center $Z\cong\finfield_q^{\times}$, is the completion of the inductive limit of the groups $\PGL_{2^n}(q)$ with respect to the normalized projective rank metric. See \cite{carderithom2018exotic} for more details on the group $A(q).$ Throughout this section, we work with the normalized projective rank metric and a suitable version of Theorem~\ref{thm:vecsp} and its corollaries.

Now we prove the main result of this section:

\begin{theorem}\label{thm:no_ex_mxd_id}
    The group $A(q)/Z$ does not admit a mixed identity.
\end{theorem}

\begin{proof}
Let
$$
w =c_0 x^{\varepsilon(1)}_{\iota(1)}c_1\cdots c_{l-1}x^{\varepsilon(l)}_{\iota(l)}c_l\in A(q) \ast\freegrp_r
$$ 
be some reduced word with constants, where $l\in\ints_{\geq2}$. We prove that $w $ has a non-trivial word image $w (A(q)^r)$.
At first, note that, by construction, we can find a large number $n = n(\varepsilon) \in\nats$, such that we can choose constants $c'_0,\dots,c'_l\in \SL_{2^n}(q)$ so that $d (c' _j,c_j)\leq\varepsilon$ for all $i\in\set{0,\ldots,l}$.
Define the word
$$
w'\coloneqq c'_0 x^{\varepsilon(1)}_{\iota(1)}c'_1\cdots c'_{l-1}x^{\varepsilon(l)}_{\iota(l)} c'_l\in \SL_{2^n}(q)\ast\freegrp_r<A(q)\ast\freegrp_r.
$$
Note that then for all $h_1,\ldots,h_r\in A(q)$, when have
\begin{align*}
    d (w'(h_1,\dots,h_r),&w(h_1,\dots,h_r))\\
    &=d (c' _0 h^{\varepsilon(1)}_{\iota(1)}c' _1\cdots c' _{l-1}h^{\varepsilon(l)}_{\iota(l)} c' _l,c_0 h^{\varepsilon(1)}_{\iota(1)}c_1\cdots c_{l-1}h^{\varepsilon(l)}_{\iota(l)}c_l)\\
    &\leq d (c' _0 h^{\varepsilon(1)}_{\iota(1)}c' _1\cdots c' _{l-1}h^{\varepsilon(l)}_{\iota(l)} c' _l,c' _0 h^{\varepsilon(1)}_{\iota(1)}c'_1\cdots c'_{l-1}h^{\varepsilon(l)}_{\iota(l)}c_l)\\
    &+\cdots + d (c' _0 h^{\varepsilon(1)}_{\iota(1)}c_1\cdots c_{l-1}h^{\varepsilon(l)}_{\iota(l)}c_l,c_0 h^{\varepsilon(1)}_{\iota(1)}c_1\cdots c_{l-1}h^{\varepsilon(l)}_{\iota(l)}c_l)\\
    &\leq\sum_{j=0}^l{d (c' _j,c_j)}\leq (l+1)\varepsilon,
\end{align*}
Hence, we obtain
$$
\abs{\diam(w'(A(q)^r)) - \diam(w(A(q)^r))}\leq 2(l+1)\varepsilon.
$$
From the definition $\norm{w'}_{\rm crit} \geq \norm{w}_{\rm crit} - \varepsilon$
and for $m \geq n$, Corollary \ref{cor:diam_wrd_img_vecsp} yields
$$\diam(w'(\SL_{2^m}(q)^r))\geq\frac{\norm{w'}_{\rm crit}}{l}-\frac{1}{2^m} \geq \frac{\norm{w}_{\rm crit}- \varepsilon}{l}  -\frac{1}{2^m}$$
and hence
$$\diam(w(\SL_{2^m}(q)^r)) \geq \frac{\norm{w}_{\rm crit}- \varepsilon}{l}  -\frac{1}{2^m} - 2(l+1)\varepsilon.$$
Thus, taking $\varepsilon>0$ small enough and $m \geq n(\varepsilon)$ large enough, we see that the word image of $w$ is non-trivial. This finishes the proof.
\end{proof}

\section*{Acknowledgments}

We are grateful to Melissa Lee for an enlightening discussion concerning Remark \ref{rmk:mult_trans}. 
The third author wants to thank the Mathematisches Forschungsinstitut in Oberwolfach for its hospitality.

\end{document}